\documentclass[11pt]{amsart}
\usepackage{amssymb,amsfonts,amsthm,amsmath,enumerate,amscd}
\usepackage{color, colordvi}
\usepackage{times}
\usepackage[english]{babel}

\newtheorem*{mainthm}{Theorem} 

\newtheorem{lem}{Lemma}[section]
\newtheorem{cor}[lem]{Corollary}
\newtheorem{defi}[lem]{Definition}
\newtheorem{thm}[lem]{Theorem}
\newtheorem{prop}[lem]{Proposition}
\newtheorem{prty}[lem]{Property}
\newtheorem{rmk}[lem]{Remark}

\theoremstyle{definition} 
\newtheorem{ex}[lem]{Example}
\newtheorem{rk}[lem]{Remark}

\newcommand{\re}{\mathbb{R}}
\newcommand{\ce}{\mathbb{C}}

\newcommand{\ee}{=}
\newcommand{\mm}{-}
\newcommand{\pp}{+}

\newcommand{\field}[1]{\mathbb{#1}}
\newcommand{\C}{\field{C}}
\newcommand{\K}{\field{K}}
\newcommand{\KP}{\field{KP}}

\newcommand{\R}{\field{R}}

\newcommand{\bbo}{{\bf 0}}

\newcommand{\bc}{{\bf c}}
\newcommand{\be}{{\bf e}}
\newcommand{\bt}{{\bf t}}
\newcommand{\bu}{{\bf u}}
\newcommand{\bv}{{\bf v}}

\newcommand{\bx}{{\bf x}}
\newcommand{\by}{{\bf y}}

\newcommand{\Cbar}{\overline{\C}}
\newcommand{\cC}{\mathcal{C}}

\newcommand{\Cn}{{\C^n}}
\newcommand{\Cp}{{\C^p}}
\newcommand{\cU}{\mathcal{U}}
\newcommand{\cV}{\mathcal{V}}

\newcommand{\clos}{{\rm clos}}

\newcommand{\dd}{{\partial}}
\newcommand{\dist}{{\rm dist}}
\newcommand{\dlt}{{\delta}}
\newcommand{\gm}{{\gamma}}
\newcommand{\Gm}{{\Gamma}}
\newcommand{\Kbar}{\overline{\K}}
\newcommand{\Km}{{\K^m}}
\newcommand{\Kn}{{\K^n}}
\newcommand{\Kp}{{\K^p}}
\newcommand{\KPn}{{\KP^n}}
\newcommand{\lbd}{{\lambda}}
\newcommand{\rd}{{\rm d}}
\newcommand{\Rbar}{\overline{\R}}
\newcommand{\Rm}{{\R^m}}
\newcommand{\Rn}{{\R^n}}
\newcommand{\Rp}{{\R^p}}

\newcommand{\ve}{\varepsilon}
\newcommand{\vp}{\varphi}

\newcommand{\Ltv}{{\rm L}}

\newcommand\normpreimage[1]{\left\lVert#1\right\rVert}
\newcommand\normproduct[1]{\left\lVert#1\right\rVert}
\newcommand{\Bif}{{\rm Bif}}

\title[Lipschitz trivial values]{Lipschitz trivial values of polynomial mappings}
\author[A. Costa]{Andr\'e Costa}
\author[V. Grandjean]{Vincent Grandjean}

\author[M. Michalska]{Maria Michalska}
\address{A. Costa, V. Grandjean, M. Michalska, Departamento de Matem\'atica, 
Universidade Federal do Cear\'a
(UFC), Campus do Pici, Bloco 914, Cep. 60455-760. Fortaleza-Ce,
Brasil}
\address{M. Michalska, Wydzia\l{} Matematyki i Informatyki, Uniwersytet 
\L{}\'o{}dzki, Banacha 22, 90-238 \L{}\'o{}d\'z{}, Poland}

\email{andrecosta.math@gmail.com - vgrandjean@mat.ufc.br -
maria.michalska@wmii.uni.lodz.pl}

\thanks{{V. Grandjean was partially supported by FUNCAP/CAPES/CNPq-Brazil 
grant 306119/2018-8}}

\subjclass[2000]{}


\begin{document}

\maketitle

\begin{abstract}
We prove that a polynomial mapping $f:\Kn\mapsto\Kp$, where $\K=\R$ or $\C$,
attains a Lipschitz trivial value~$\bc$
if and only if there exist a polynomial 
mapping $g:\K^m\mapsto\Kp$, for which the value~$\bc$ is a regular value of 
properness, 
and a linear surjective projection $\pi:\Kn\to\K^m$ such that $f = g\circ \pi$. 
The integer $m$ is the $\K$-codimension of the accumulation set at infinity of
the level $f^{-1}(\bc)$ in the hyperplane at infinity. In the complex case, it
is equivalent to require the mapping $g$ be generically finite and dominant.
Last, we show this result cannot extend to rational mappings over $\Kn$.
\end{abstract}

\section*{Introduction}

Let $f:\Kn\to\Kp$ be a dominant polynomial mapping over $\K = \R$ or $\C$. 
There exists a smallest subset $\Bif_k(f)$, contained in an algebraic 
subset of $\Kp$ of positive codimension, with the following property: the 
mapping $f$ induces a locally trivial $\cC^k$ fibre bundle structure over each 
connected component of $\Kp\setminus \Bif_k(f)$. The notion of $\cC^\infty$ 
triviality is generally finer than that of $\cC^0$ triviality and over the 
last fifty years, following the seminal paper \cite{Tho}, a 
significant literature about 
the $\cC^0$ and the $\cC^\infty$ local triviality has been developed
(see for instance
\cite{Ver,Har,Rab,KOS,JeKu}).

This paper investigates an intermediate case: when and over which 
subset of values the mapping induces a locally 
bi-Lipschitz trivial fibre bundle structure. Our goal is to characterize 
polynomial mappings admitting \em Lipschitz trivial values, \em that is over
a neighbourhood of which there is a bi-Lipschitz trivialization, problem 
recently raised in~\cite{FGS}. Our main result 
Theorem~\ref{MainThm} when combined with Proposition \ref{prop:lip-triv-reg}
implies the following 
\begin{mainthm} 
Let $f:\Kn\to \Kp$ be a polynomial mapping and let $n \mm 1 \mm m$ be
the dimension of 
the set of accumulation points at infinity of the fibre $f^{-1}(\bc)$. 
The mapping $f$ attains 
the Lipschitz trivial value $\bc$
if and only if 
$$
f = g \circ \pi,
$$
for a linear surjective projection $\pi:\Kn\to\Km$ and a polynomial mapping 
$g :\Km\to\Kp$ for which $\bc$ is a regular value of properness.  
\end{mainthm}
In the complex case, the statement equivalently requires that $m=p$ and
the polynomial mapping $g~:~\Cp~\to~\Cp$ be dominant and generically finite (see 
Corollary~\ref{corComplexLtv}). Therefore, either almost all values of a 
complex polynomial mapping are Lipschitz trivial or there are none. 
In contrast with the complex case, there exist non-proper real polynomial
mappings admitting values of properness. 

The main result completely describes the real and complex polynomial mappings 
admitting Lipschitz trivial values, recovering the case of complex 
polynomial functions of \cite{FGS}. Our proof is very different, more concise, 
and covers both real and complex cases. Moreover, we show that the 
main theorem cannot extend without further hypotheses to a wider class of
rational mappings (even those which are regulous, see~\cite{4guys}). 

The article is organized as follows. 
Section~\ref{secGeneralLipPrties} presents properties of Lipschitz trivial values 
of differentiable mappings. The main result is proved in Section~\ref{secMain} 
while Section~\ref{secLipRealComplex} describes the set of Lipschitz trivial 
values in the real and complex cases. Section~\ref{SecComplexification} deals with 
the relation 
between Lipschitz trivial values of a real mapping and that of its 
complexification. In  Section~\ref{secRational} we show that the main result 
cannot extend to rational mappings.

\section{Preliminaries}

Throughout the paper $\K \ee \ce$ or $\re$.
We use the conventions that $\dim \emptyset \ee \mm 1$ and empty mappings are 
proper and continuous. 
\\
We  embed the affine space~$\Kn$  in  $\KPn$ as $\bx \mapsto [\bx:1]$. The 
hyperplane at infinity $H_\infty := \KPn\setminus\Kn$ consists of the points 
of $\KPn$ of the form $[\bv:0]$.

\begin{defi}
The \em accumulation set at infinity $X^{\infty}$ of a subset $X$ of $\Kn$ \em is 
defined as
$$
X^{\infty}:\ee \overline{X}^{\KPn}\cap H_{\infty},
$$
where $ \overline{X}^{\KPn}$ is the closure
of $X$ taken in $\KPn$.
\end{defi}

The following notions lie at the heart of our problem.

\begin{defi}
A mapping $\vp:\Kn\to\Kp$ is proper at the value $\bc\in\Kp$, if there exists a 
neighbourhood $\cV$~of~$\bc$ in $\Kp$ such that the restriction mapping 
of $\vp$ to $\vp^{-1}(\cV)$ is proper.  
Let us denote by $J(\vp)$ the  \em Jelonek set \em of values at which the mapping $\vp$ 
is not proper.
\end{defi}

The Jelonek set of a real or complex polynomial mappings is always contained in
an algebraic set of dimension at most $n\mm 1$, see~\cite{Je99,JeKu}.

\begin{defi}\label{defTOPtriv}
A mapping $\vp:\Kn\to\Kp$ \em is  topologically trivial at the value
$\bc \in\Kp$, \em  if there exist a neighbourhood $\cV$ of $\bc$ in $\Kp$ and a 
trivializing homeomorphism 
\begin{equation}\label{eqnfibration}
H: \vp^{-1}(\bc) \times \cV \to \vp^{-1}(\cV)
\end{equation}
which  satisfies $(\vp \circ H) (\bx,\bt) = \bt$.

When $H$ is a $\mathcal{C}^\infty$ diffeomorphism, the mapping $\vp$ is called 
\em $\cC^\infty$ trivial at the value $\bc$.  \em The complement of the set of 
values at which the mapping $\vp$ is  $\cC^\infty$ trivial is 
the set of bifurcation values $\Bif(\vp)$, see~\cite{Tho}. 
\end{defi}

\begin{rk}\label{rk:empty}
In particular, the mapping $\vp$ is locally trivial at any value
of the open subset $\Kp\setminus\clos(Im(\vp))$, the complement of 
the closure of the image of $\vp$. 
\end{rk}

\begin{ex}
Note that the real polynomial function $x \mapsto x^{2021}$ is topologically 
trivial at each $c\in\R$, but not $\cC^\infty$ trivial at $0$.
\end{ex}

Let $K_0(\vp) :\ee \vp (\text{crit} (\vp))$ be the set of critical values of 
$\vp$. The mapping $\vp$ is said to be regular at $\bc$ once $\bc\notin K_0(\vp)$. 
When a polynomial mapping $\vp$ is proper and regular at a value $\bc$, 
the mapping~$\vp$ is $\cC^\infty$ trivial at $\bc$ by Sard's and Ehresmann's 
Theorems. Thus the set $\Bif(\vp)\setminus K_0(\vp)$ contains only non-proper 
values when it is non-empty. Moreover,  by~\cite{JeKu}, when $\vp$ is 
polynomial and $n\ee p$ we get 
$\Bif (\vp) \ee K_0(\vp)\cup J(\vp)$.

\section{Lipschitz trivial values}\label{secGeneralLipPrties}

In  this section we present general properties of differentiable mappings with 
Lipschitz trivial values. 

Let again $\vp:\Kn\to\Kp$ be a mapping.

Any subset $X$ of $\Kn$ inherits a metric structure from restricting the ambient 
euclidean distance to $X$, this metric is called \em the outer metric 
on $X$. \em  On a product $X\times Y$ of
metric spaces we will consider the product metric.

\begin{defi}\label{defLipschitztriv}
A \em Lipschitz trivial value $\bc \in \Kp$ \em of the mapping $\vp$ is a value
such that there exists a trivializing homeomorphism as in~\eqref{eqnfibration} 
of Definition~\ref{defTOPtriv} which is bi-Lipschitz with respect to the outer
metrics. 
Denote by $\Ltv(\vp)$ \em the set of Lipschitz trivial values of~$\vp$. \em 
\end{defi}
 
\begin{rk}
In light of Definition \ref{defLipschitztriv} and Remark \ref{rk:empty}, the 
subset $\Ltv(\vp)$ is open and contains $\Kp\setminus\clos(Im(\vp))$.
In particular, the mapping $\vp$ attains a Lipschitz 
trivial value only if its image is Zariski-dense in $\Kp$.
\end{rk}

\begin{ex}\label{ex-simple}
The set of Lipschitz trivial values of the polynomial mapping $\K^3\to\K^2$ 
defined as $(x,y,z) \mapsto (x,xy + xz)$ is
$\K^2\setminus 0\times \K$. There is a single critical value $(0,0)$, and none
of the 
values $(0,b)$ with $b\neq 0$ is taken. Moreover, 
each level $(a,b)$ with $a\neq0$ is an affine line.
\end{ex}

The next result emphasizes the rigid asymptotic behaviour of  levels  near a 
Lipschitz trivial value.

\begin{prop}\label{lemfLipschitz} 
Assume that $\bc$ is a Lipschitz trivial value of $\vp:\Kn\to\Kp$. There exists a neighbourhood
$\cV$ of $\bc$ such that the following properties hold: 
\begin{enumerate}[(i)]
\item \label{pointLipschitz} the mapping $\vp$ is Lipschitz on 
$\vp^{\mm 1}(\cV)$;
\item \label{pointTube} there exist $0<\delta<\epsilon$ such that 
$$ 
T_{\delta}(\vp^{\mm 1}(\bc))\subset \vp^{\mm 1}(\cV) \subset  
T_{\epsilon}(\vp^{\mm 1}(\bc)),
$$
\end{enumerate}
where the open tube $T_{r}(S)$ of radius $r$ around a subset $S$ of $\Kn$ 
is defined as 
$$
T_{r}(S):\ee \{ \bx\in \Kn : \dist (\bx,S)<r    \}.
$$
\end{prop}
\begin{proof}
Denote by $d_\bc$  the outer metric on $\vp^{\mm 1}(\bc)$ and by 
$\cU := \vp^{\mm 1}(\cV)$.
Since $\vp$ has a Lipschitz trivial value at $\bc$, there exist an open ball 
$\cV\ee B^p(\bc,r)$ of $\Kp$ and a bi-Lipschitz homeomorphism 
$$
G= (\vp, \psi): \cU \mapsto \cV \times \vp^{-1}(\bc).
$$ 
Therefore, there exists $L > 1$ such that for any $\bx,\bx' 
\in \cU$ we have 
\begin{eqnarray}
\label{eqnbilipschitz}
 \frac{1}{L} \normpreimage{\bx-\bx'}  \leq \normproduct{G(\bx) - 
G(\bx')} \leq L \normpreimage{\bx-\bx'} .
\end{eqnarray}

Point (i) follows from inequalities \eqref{eqnbilipschitz} since
for all $\bx,\bx' \in \cU$ we find 
$$
\| \vp(\bx) - \vp(\bx') \| \leq \| \vp(\bx) - \vp(\bx') \| + 
d_\bc(\psi(\bx),\psi(\bx'))  \leq L \normpreimage{\bx-\bx'}.
$$

To prove (ii), define the following radii 
$$
\delta :\ee \frac{r}{L} \;\; {\rm and} \;\; \epsilon:\ee L r.
$$
Point (i) yields  the first inclusion $T_{\delta}(\vp^{\mm 1}(\bc)) 
\subset \cU$. 
Note that for $\bc,\bt\in \cV$ and $\bx'\in \vp^{-1}(\bt)$, there exists $\bx 
\in \vp^{-1}(\bc)$ such that $\psi(\bx) = \psi(\bx')$. Therefore, 
estimates~\eqref{eqnbilipschitz} provide 
$$
\dfrac{1}{L} \normpreimage{\bx-\bx'} \leq  \normproduct{G(\bx) - G(\bx')} 
= \|\bc \mm \bt\| +  d_\bc(\psi(\bx),\psi(\bx'))   = \|\bc \mm \bt\|  .
$$
Thus we obtain $\cU\subset  T_{\epsilon}(\vp^{\mm 1}(\bc))$.
\end{proof}
\begin{rmk}\label{rk:partial-bounded} 
Point (i) of Proposition \ref{lemfLipschitz} implies that each first order 
partial derivative of each component of the mapping $\vp$ is bounded over 
$\vp^{-1}(\cV)$.
\end{rmk}

\begin{prty}\label{prtySuspension}
Let $\tau:\Km\to \Kp$ be a mapping Lipschitz trivial at the value $\bc$ and $\pi:\Kn\to\Km$ be a linear surjective projection. Then the mapping
$\tau\circ\pi:\Kn\to\Kp$ is Lipschitz trivial at the value $\bc$. 
\end{prty}
\begin{proof}
Let $\vp := \tau\circ\pi$. Up to a $\K$-linear change of coordinates in $\Kn$, 
we can assume that for any subset $V$ of $\Kp$ the following holds true
$$
\vp^{\mm 1}(V)\ee \tau^{\mm 1}(V)\times \K^{n\mm m}.
$$
Denote $(\bu,\bv) \in \Km \times \K^{n\mm m} = \Kn $. 
If the bi-Lipschitz homeomorphism
$
G : \tau^{-1}(\cV) \mapsto \cV \times \tau^{-1}(\bc) 
$ 
provides trivialisation of $\tau$ over a neighbourhood $\cV$ of $\bc$, then
the mapping  
$$
H : \vp^{-1}(\cV) \to \cV\times \vp^{-1}(\bc), \quad (\bu,\bv) \to (G(\bu), 
\bv)
$$
is a bi-Lipschitz homeomorphism trivialising $\vp$ over $\cV$. 
\end{proof}

\begin{prty}\label{prop:properness-lipschit}
Let $\vp:\Kn\to\Kp$ be a smooth  mapping with a nowhere dense set of critical 
values.  Any regular value of $\vp$ that is also a  value of properness 
is a Lipschitz trivial value.
\end{prty}
\begin{proof}
Let $\cV$ be a non-empty open subset of $\Kp$ such that $\vp$ is proper
over $\cV$. 
We can further assume that $\clos(\cV)$ does not intersect with $K_0(\vp)$.
Therefore, the mapping $\vp$ is $\cC^\infty$ locally trivial over $\cV$ by 
Ehresmann's Theorem.
The restriction of any $\cC^\infty$ trivialisation of~$\vp$ over $\cV$
to any open subset $\cU$ relatively compact in $\cV$ is necessarily 
bi-Lipschitz over~$\cU$.
\end{proof}
\begin{prop}\label{prop:lip-triv-reg}
Let $\vp:\Rn\to\Rp$ be a 
$\cC^k$ mapping with $k\geq \max(n-p+1,1)$. 
Any Lipschitz trivial value of $\vp$ is a regular value.
\end{prop}
The proof will follow from the next result.
\begin{lem}\label{lem:lip-triv-reg} 
Let  $f: \cU \to \cV$ be a $\cC^1$ mapping, where $\cU$ and $\cV$ are open
subsets of $\Rn$ and $\Rp$ respectively. If there exists a bi-Lipschitz 
homeomorphism 
$$
(f,\psi) : \cU \to \cV \times F
$$
where $F$ is a $C^1$ sub-manifold of some $\R^q$,
then $f$ has no critical points in 
$\cU$.
\end{lem}
\begin{proof}
Let $H:\cV\times F \to \cU$ be the inverse of the bi-Lipschitz homeomorphism
$(f,\psi)$. Since $H$ is bi-Lipschitz, it is differentiable almost everywhere. 
At a point $\by_0$ where $H$ is differentiable, its differential $D_{\by_0} H$ has 
rank $n$. 
Moreover, having $H$ bi-Lipschitz implies that any limit of the form $D = 
\lim_{\by_n \to \by}  D_{\by_n}H$ at the 
given point $\by$ of $\cV\times F$ has also rank $n$. 
Therefore the differential of $f$ must have rank $p$ at each point of $\cU$.
\end{proof}
\begin{proof}[Proof of Proposition \ref{prop:lip-triv-reg}] 
By hypothesis $\vp$ satisfies Sard's Theorem. 
Let $\bc$ be a Lipschitz trivial value of $\vp$.
\\
If $\bc$ does not lie in $Im(\vp)$, the image of $\vp$, then it belongs to 
$\Rp\setminus\clos(Im(\vp))$, thus is a regular value.
\\
Assume $\bc$ is a value taken by $\vp$. Let $\cV$ be an open neighbourhood of $\bc$
over which $\vp$ is Lipschitz trivial.
Thus, by Lemma \ref{lem:lip-triv-reg}, 
the mapping $\vp$ has no critical point in $\vp^{-1}(\cV)$, 
therefore $\bc$ is a regular value.
\end{proof}

\begin{prty}\label{cor:asymptotic-set}
Let $\vp:\Kn\to\Kp$ be a  mapping locally  Lipschitz trivial over the 
connected open subset $\cV\subset \Kp$.  
Then the family $(\vp^{\mm 1}(\bt)^\infty)_{\bt\in \cV}$ of accumulation 
sets at infinity of the levels of $\vp$ is constant, i.e.
$$
\vp^{\mm 1}(\bt)^\infty\ee \vp^{\mm 1}(\cV)^\infty, \;\; \forall \bt \in \cV.
$$
\end{prty}
\begin{proof} Consider two sequences $(\bx_k)_k$ and $(\bx_k')_k$ of $\Kn$ 
satisfying 
the following property: there exists a positive constant $A$ such that
$$
|\bx_k - \bx_k'| \; \leq \; A \;\, {\rm for} \;\, k \gg 1.
$$
If furthermore $|\bx_k|$ goes to $\infty$ and does so such that $[\bx_k:1] \to 
[\lbd:0] \in H_\infty$ as $k \to \infty$, then we deduce that $|\bx_k'|$ goes
to $\infty$ and $[\bx_k':1] \to [\lbd:0]$ as $k$ goes to $\infty$.

Let $\bc\in\cV$ be a Lipschitz trivial value of $\vp$. 
Up to taking a smaller $\cV$ containing $\bc$, 
point (ii) of Proposition~\ref{lemfLipschitz} states that 
$\vp^{-1}(\cV)$ is contained in the open tube $T_\ve (\vp^{-1}(\bc))$ for some 
positive radius $\ve$. The first part of the proof then gives the result.
\end{proof}

\section{Polynomial mappings with Lipschitz trivial values depend on fewer variables}\label{secMain}

Let $f:\Kn\to\Kp$ be a polynomial mapping.
The level $f^{-1}(\bt)$ is denoted by $F_\bt$.

\begin{thm}\label{MainThm}
Let $f:\Kn\to \Kp$ be a polynomial mapping with $\dim F_\bc^\infty \ee n-1-m$
for a value $\bc\in\Kp$. If the mapping $f$ attains $\bc$ as a Lipschitz 
trivial value, then there exist a polynomial mapping $g:\Km \to \Kp$
which is proper at $\bc$ and a linear surjective projection $\pi:\Kn\to\Km$
such that
$$
f\ee g\circ \pi.
$$
\end{thm}

We start with the following key result.

\begin{lem}\label{lemNotDepend}
Let $f:\Kn\to\Kp$ be a polynomial mapping and assume that the point 
$[1:0:\dots:0]$ lies in $F_\bc^\infty$. If there exists a neighbourhood $\cV$
of $\bc$ such that $f$ satisfies points (i) and (ii) of
Proposition~\ref{lemfLipschitz}, then the mapping $f$ does not depend on 
the coordinate $x_{1}$. 
\end{lem}
\begin{proof}
Since $[1:0:\dots:0]\in F_\bc^\infty$, 
there exists an arc $\gm: I \to F_\bc$  parametrized as 
$$
\gm (t) = (t^d, p(t)+A_0(1/t)) \in \K \times \K^{n-1},
$$ 
where $I$ is a connected component of the complement of an Euclidean ball 
of $\K$, the 
mapping $p~:~\K~\to~\K^{n\mm 1}$ is polynomial of degree $\leq d-1$, and $A_0$ 
is a $\K$-analytic map germ $(\K,0) \to (\K^{n-1},\bbo)$.

Consider the following dominant polynomial mapping 
$$
\Gm : \K\times\K^{n-1} \to \Kn, \;\; (t,\ve) \mapsto (t^d,p(t) + \ve).
$$
Let $\gm_\ve :\K \to \Kn$ be the polynomial arc 
$$
\gm_\ve:t\mapsto \gm_\ve(t) := \Gm(t,\ve).
$$
Since $A_0(1/t)\to \bbo$ as $|t|\to \infty$, we conclude that 
$$
\|\gm_\ve(t) - \gm (t)\| = \left\|\ve - A_0(1/t)\right\| \to \|\ve\|. 
$$
Take a neighbourhood $\cV$ of $\bc$ in $\Kp$ such that $f$ satisfies 
points (i) and (ii) of Proposition~\ref{lemfLipschitz}.
Point (ii) of Proposition~\ref{lemfLipschitz} and the definition of $\Gm$
guarantee the existence of constants $\delta>0$ and $R>0$ such that 
\begin{eqnarray}\label{eqnGammainPreimage}
\gm_\ve(t)\in f^{-1}(\cV)
\end{eqnarray}
for any $\|\ve\|<\dlt$ and $t\in I$ such that $|t|>R$.
Since $f$ is Lipschitz on  $f^{-1}(\cV)$, we get
$$
\| f(\gm_\ve(t)) - \bc\| \leq L \cdot \|\ve - A_0(1/t)\| \to L\cdot  
|\vert \ve\|
$$
as $|t| \to \infty$. Therefore for each $\|\ve\| < \dlt$, the polynomial 
mapping $t\to f\circ\gm_\ve(t)$ is bounded, thus constant. 
Writing $\bx = (x_1,\by)$, we deduce 
$$
\bbo \equiv \frac{\rd}{\rd t} (f\circ\gm_\ve) (t) = 
d \cdot t^{d-1}\cdot\dd_{x_1} f (\gm_\ve(t)) + \dd_{\by}f(\gm_\ve(t)) \cdot p'(t).
$$
Since $p$ has degree $\leq d-1$ 
and the first order partial derivatives of $f$ are bounded along 
$\gm_\ve$ by Remark \ref{rk:partial-bounded},
for $\|\ve\| < \dlt$ we conclude that 
$$
(\dd_{x_1} f) \circ \gm_\ve \equiv 0.
$$
Since the subset $\K \times \{\ve:\|\ve\| < \dlt\} \subset 
\Kn$ is open and non-empty and the mapping $\Gm$ is dominant, we conclude
that the mapping $\dd_{x_1} f$ is identically null. 
\end{proof}

\begin{proof}[Proof of Theorem~\ref{MainThm}]
Note that $f$ satisfies the claim of Proposition~\ref{lemfLipschitz}. If 
$\dim F_\bc^\infty\ee \mm 1$, then the fibre $F_\bc$ is compact and from
point (ii) of Proposition~\ref{lemfLipschitz} the subset $f^{-1}(\cV)$ is
compact for  a small compact neighbourhood $\cV$ of $\bc$. Thus $f$ is proper
at $\bc$ and taking $\pi$ as the identity mapping of $\Kn$ yields the claim.

Assume $F_\bc^\infty$ is of dimension $n-1- m\geq 0$.  Thus there exist $n-m$ 
points $[\bv_1:0], \ldots, [\bv_{n-m}:0]$ of~$F_\bc^\infty$ such that the
vectors $\bv_1,\dots, \bv_{n-m}$ are $\K$-linearly independent in $\Kn$. 
Take a $\K$-linear change of coordinates $\ell:\Kn\to\Kn$ such that $\ell(\bv_j)
\ee \be_j$ for $j=1,\ldots,n-m$, where $\{\be_1, \dots, \be_n\}$ is the 
standard orthonormal basis of $\Kn$. Applying Lemma~\ref{lemNotDepend} 
 we conclude that the polynomial mapping $f\circ \ell$ depends 
only on $\bu := (x_{n-m+1},\dots, x_n)$. 
Let $g$ be the polynomial mapping restriction of $f\circ\ell$ to $\Km$,
the subspace of $\Kn$ generated by $\be_{n-m+1},\ldots,\be_n$.
Let $\pi_0 :\Kn\to \Km$ be the orthogonal projection of $\Kn$ onto  the 
subspace $\Km$. Therefore, we find
$$
f = g \circ \pi_0 \circ \ell^{-1}.
$$

Note that  $(f\circ \ell)^{\mm 1}(\bt)\ee \K^{n-m}\times g^{\mm 1}(\bt)$. 
Since
$$
n-1-m = \dim F_\bc^\infty = n-m + \dim g^{-1}(\bc)^{\infty},
$$
we deduce that $g^{-1}(\bc)$ is compact. From Proposition~\ref{lemfLipschitz} 
applied to the levels of $f\circ \ell$ over $\cV$ we get that 
$g^{\mm 1}(\cV)$ is bounded. Thus $\bc$ is a value of properness of $g$. 
\end{proof}

\section{On the set of Lipschitz trivial values of real and complex mappings}\label{secLipRealComplex}

This section presents some consequences of Theorem~\ref{MainThm}. In particular, 
complex  polynomial mappings admitting Lipschitz trivial values have a very 
rigid structure, while the real setting allows for more variety.

\begin{cor}\label{corComplexLtv}
A complex polynomial mapping  $f:\Cn\to \Cp$ attains a Lipschitz trivial value 
if and only if there exist a dominant polynomial mapping $g:\Cp\to \Cp$ 
and a linear surjective projection $\pi:\Cn\to\Cp$ such that
$$
f\ee g\circ \pi.
$$
In such a case we get
$$
\Ltv(f) = \ce^p\setminus {\rm Bif}(g).
$$
In particular, the set of regular Lipschitz trivial values is either empty or
the complement of an 
algebraic hypersurface.
\end{cor}
\begin{proof}
Assume there exists a dominant polynomial mapping $g:\Cp\to\Cp$ such that 
$
f\ee g\circ \pi
$ 
for some linear surjective projection $\pi:\Cn\to\Cp$. Therefore $g$ is 
generically finite and by Properties~\ref{prtySuspension} 
and~\ref{prop:properness-lipschit}, the set $\Ltv(f)$ of Lipschitz trivial values 
is not empty.
For  the converse statement, note that $f$ is dominant and $n\mm p\ee \dim 
F_\bc \ee 1 \pp \dim F_\bc^\infty $ for a generic level~$\bc$ of $f$, so Theorem~\ref{MainThm} gives the claim. 

For the second part of the assertion, observe that the subset 
$\Ltv(g)\cap 
(J(g)\cap K_0(g))$ is empty by Property~\ref{prtySuspension} and 
Proposition \ref{prop:lip-triv-reg},
since $g$ is generically finite. 
We recall that $J(g)\cup K_0(g)\ee \Bif(g)$ and if non-empty, it is an 
algebraic hypersurface by~\cite{JeKu}.
\end{proof}

As a consequence, and with a different proof, we recover the main result 
of~\cite{FGS}.
\begin{cor}[\cite{FGS} Theorem 10]
A complex polynomial function $f:\Cn\to\C$ admits a Lipschitz trivial value if and 
only if it depends on a single variable, i.e., there exist $(n-1)$ linearly 
independent vectors $\bv_2, \ldots, \bv_n,$ of $\Cn$ such that 
$\dd_{\bv_i} f \equiv 0$ for all $i\geq 2$.
\end{cor}

The real case is more nuanced than the complex one and Lipschitz trivial values 
admit a richer structure as we can see below.
\begin{cor}\label{corLocGlobReal}
Let $f:\Rn\to \Rp$ be a polynomial mapping admitting a Lipschitz trivial value. 
There exists a polynomial mapping $g:\Rm \to \re^p$ and linear surjective 
projection $\pi :\Rn\to \Rm$ such that $f\ee g\circ \pi$ and
$$
\Ltv (f) = \Rp\setminus (J(g)\cup K_0(g)).
$$
Moreover, the mapping $g$ is unique up to linear changes of coordinates. 
\end{cor}

\begin{proof}
The demonstration is similar to proof of Corollary~\ref{corComplexLtv}.
Moreover, up to a linear change of coordinates in $\Rn$,
we have $Df = Dg\oplus\bbo : \Rm\times\R^{n-m}\to\Rp$. 
Thus using Property~\ref{prtySuspension}, Proposition \ref{prop:lip-triv-reg}
and Theorem~\ref{MainThm} we get 
$$
\Ltv(f) = \Ltv(g).
$$
To show uniqueness take $g:\Rm\to\Rp$ proper at $\bc$ such that $f\ee g\circ 
\pi$ for a linear surjective mapping $\pi :\Rn\to \Rm$. For any polynomial 
mapping $h:\R^k\to\Rp$ and a linear surjective projection $\sigma:\Rn\to \R^k$
such that $f\ee h\circ\sigma$, we get, up to a linear change of coordinates 
$\ell:\re^k\to\re^k$, that
$$ 
(h\circ\ell)^{\mm 1}(\bt) \ee g^{\mm 1}(\bt)\times \R^{k\mm m}
$$
for $\bt\in \re^p$, since at least the level $\bc$ of $g$ is compact. Thus
either $m < k$ and~$h$ does not attain a proper value, or $m\ee k$ and 
$h\circ\ell \ee g$.
\end{proof}

\begin{ex}\label{exMotzkin}
Let $f: \R^3 \mapsto \R$ be the suspension at infinity of the Motzkin polynomial 
given by 
$$
f(x,y,z) =x^4y^2+x^2y^4-3x^2y^2+1.
$$

We have $J(f)\ee [1,\infty)$ and $K_{0}(f)\ee \{0,1\}$. Moreover,  
$$
\Ltv(f) \ee (\mm \infty,1) \setminus \{0\}.
$$ 
Indeed, the values of $[1,+\infty)$ are not Lipschitz trivial values of $f$, since $f$ 
does not satisfy the necessary condition (ii) of Proposition~\ref{lemfLipschitz} 
(the distance between any two levels in $(1,\pp \infty)$ is zero). 
\end{ex}
Example~\ref{exMotzkin} illustrates that the set of Lipschitz trivial values of 
a real polynomial mapping can be open and not dense in the image, whereas for 
complex mappings Lipschitz trivial values
follow a local-global principle as stated in Corollary \ref{corComplexLtv}.

For polynomial mappings Property~\ref{cor:asymptotic-set} is 
refined as the following necessary condition on the fibres.
\begin{prty}~\label{prtyConeAtInfty}
Let $f:\Kn\to\Kp$ be a polynomial mapping. There exists a $\K$-linear subspace 
$A$ of $\Kn$ of positive codimension such that
$$
\widehat{F_\bt^\infty} \ee A, \quad \text{for all } \bt\in \Ltv(f)\cap Im(f),
$$
where $\widehat{F_\bt^\infty}$ is the $\K$-cone of $\Kn$ over $F_\bt^\infty$
with vertex at the origin, and where the cone over the empty set is defined as
the null subspace.
\end{prty}

\begin{ex}\label{ex-bad}
The polynomial mapping $f:\K^3\to\K^2$, defined as $(x,y,z) \mapsto (x,xy + z)$, 
is surjective and $\cC^\infty$ trivial at each $\bc\in\K^2$.
Each level $F_\bt$ is an affine line, like the mapping in Example 
\ref{ex-simple}. Yet, the family of accumulation sets at infinity 
$(\widehat{F_\bt^\infty})_{\bt\in\K^2}$ is nowhere locally constant. Therefore
Property~\ref{prtyConeAtInfty} implies that this mapping cannot admit
any Lipschitz trivial value. 
\end{ex}

\section{Lipschitz trivial values of real mappings and their complexifications}\label{SecComplexification}
Throughout this section let $f:\Rn\to\Rp$ be a real polynomial mapping and let 
$f_\C$ be its complexification. We will show  that the set of real Lipschitz trivial values of 
$f_\C$ is either empty or equal to the set of Lipschitz trivial values of $f$, up 
to a set of positive codimension.

\begin{prop} \label{propComplexificationValuesInclusion}
If $f_\C$ admits a Lipschitz trivial value, then $\Ltv(f)$ is a semi-algebraic 
dense open subset of~$\Rp$.
More precisely, 
$$
\Ltv(f_\C)\cap \Rp \   \subset \  \Ltv(f).
$$
\end{prop}
\begin{proof}
Denote $F_{\C,\bt} : =f_\C^{-1}(\bt)$.
If $\Ltv(f_\C)\cap Im(f_\C)$ is not empty, then $f_\C$ is dominant, and thus
the image of $f$ is Zariski dense.
Corollary~\ref{corComplexLtv} and Property~\ref{prtySuspension} imply 
the existence of a $\C$-linear
subspace $A_\C$ of $\Cn$ of dimension $n-p$ such that for each $\bt\in 
\ce^p\setminus \Bif(f_\C)$, 
the level $F_{\C,\bt}$ is a disjoint union of finitely many affine subspaces 
of $\C$-dimension $n-p$, parallel to $A_\C$. Since $f(\Rn)$ is Zariski dense, 
there exists an open set $\cV$ in $\re^p$ such that the level $F_{\bt}$ is of 
dimension $n\mm p$ for $\bt\in \cV$. As the intersection of the complex fibre 
$F_{\C,\bt}$ with $\re^n$, the fibre $F_{\bt}$ is necessarily a disjoint union of 
parallel real affine subspaces.
Therefore, $\dim f^{\mm 1}(\cV)^\infty \ee n\mm 1 \mm p$. Note that $f$, as the 
restriction of $f_{\ce}$ to $\re^n$, satisfies assumptions (i) and (ii) of 
Lemma~\ref{lemNotDepend}.
Using Theorem~\ref{MainThm} we get that $f\ee g\circ \pi$ for some real linear 
surjective projection $\pi:\re^n\to\re^p$ and real polynomial mapping 
$g:\re^p\to\re^p$. Thus $f_{\ce}\ee g_{\ce}\circ \pi_{\ce}$  and necessarily 
$g_{\ce}$ is generically finite. 
Since  Corollaries~\ref{corComplexLtv} and~\ref{corLocGlobReal} yield
$$
\Ltv(f_{\ce}) =  \ce^p\setminus \Bif (g_{\ce}), \quad \text{and}\quad 
\Ltv(f) =  \re^p\setminus \Bif (g) .
$$
We get the claim since $\Bif(g) = J(g) \cup K_0(g)$ is semi-algebraic of
positive codimension by \cite{JeKu}.
\end{proof}

\begin{prop}
The subset $B:= \Bif (g_\C) \cap \Rp\setminus \Bif (g)$ is semi-algebraic
of dimension $\leq p-1$. Either
$$
\Ltv(f)\setminus B\ee \Ltv (f_{\ce})\cap \re^p
\;\;
{\rm when} \; \dim F_\bc^\infty 
\ee n\mm 1\mm p \;\; \text{\rm for some }\bc\in \Ltv(f),
$$ 
or	 $\Ltv(f_\C)\ee \emptyset$ 	 otherwise.	 

\end{prop}
\begin{proof} In view of Proposition~\ref{propComplexificationValuesInclusion}
we only need to show the inclusion 
$$
\Ltv(f)\setminus B \subset \Ltv (f_{\ce}).
$$
If $\Ltv(f)\neq \emptyset$, then Property~\ref{prtyConeAtInfty} 
holds true and let $m:= \dim A$. From Corollary~\ref{corComplexLtv} the 
condition $m=n-p$ is necessary to have $\Ltv(f_{\ce})\neq \emptyset$. When
$m\ee n\mm p$, there exist a polynomial mapping $g:\Rp\to\Rp$ with Zariski 
dense image and a linear surjective projection $\pi:\Rn\to\Rp$ 
such that $f =g\circ\pi$. Thus $f_{\ce} = g_\C \circ \pi_\C$, and $g_\C$ is
dominant. 
We have $\Bif (g)\cup B \ee \Bif (g_\C) \cap\Rp$ and the claim follows from  
Corollaries~\ref{corComplexLtv} and~\ref{corLocGlobReal}.
\end{proof}

\section{Rational functions and Lipschitz trivial values}

\label{secRational}

It is natural to ask whether we can extend the category of mappings that satisfy 
the claim of our main result. We answer in the negative, as 
Property \ref{propexampleRationalLipchitzManyVar} demonstrates that 
Theorem~\ref{MainThm} is sharp in the sense that it does not
hold for rational but non-polynomial mappings. 

Let $f:\Kn \dashrightarrow \K$ be a rational function, $n\geq 2$. 
Its indeterminacy locus $I(f)$  is the subset of $\Kn$ where denominator and 
numerator vanish simultaneously (for all representations of $f$ as a fraction).
Let $\Kbar$ be the compactification of $\K$ defined as follows
$$
\Cbar := \C \cup \{\infty\} \;\; {\rm and} \;\; 
\Rbar := \R \cup\{\pm\infty\}.
$$

\begin{prty}\label{prtyIndetLocNonempty}
Assume that the rational function $f:\Kn\dashrightarrow\K $ does not extend 
continuously through the point $\bx_0 \in \Kn$, i.e., the subset $J:\ee 
\{\lim_{\bx\to \bx_0} f(\bx) \} \subset \Kbar$ of accumulation values 
of $f$ at $\bx_0$ does not reduce to a single value in $\K$. Then
$\Ltv(f)\cap J\ee \emptyset$.
\end{prty}
\begin{proof}
If $\K =\C$ then $J = \Cbar$. When $\K= \R$, the set $J$ is closed,  
semi-algebraic and has non-empty interior.
Suppose $J\cap \Ltv(f)$ is non-empty, thus it is open. In such a case, 
there exists an open subset $\cV$ of $J\cap \Ltv(f)$ such that $f$ is Lipschitz 
trivial over $\cV$. Then any  trivializing bi-Lipschitz
homeomorphism satisfies Estimates~\eqref{eqnbilipschitz}, contradicting 
the fact that $\bx_0$ lies in the closure of any level $f^{-1}(t)$ when $t\in\cV$.
\end{proof}

\begin{cor}
A complex rational function with Lipschitz trivial values has empty 
indeterminacy locus. 
\end{cor}
On the other hand, the real setting is more flexible. Real rational 
functions may extend continuously (or even smoothly) through their 
indeterminacy locus onto $\Rn$, in such a case they are called regulous.

\begin{prop}\label{propexampleRationalLipchitzManyVar}
There exist rational functions $f:\Kn \dashrightarrow \K$ with empty 
indeterminacy locus that admit Lipschitz trivial values which are not
values of properness, and are never of the form $g\circ\pi$ with 
$g:\Km\dashrightarrow\K$ a rational function and $\pi:\Kn\to\Km$ a 
linear surjective projection with $n>m$.  
\end{prop}
\begin{proof}
Let $h:\re^{n-1}\to\re$ be the non-constant 
function $\bx \mapsto h(\bx) :=1+\sum_{i=1}^{n-1} x_i^2$
and consider the rational $\cC^\infty$ function $f:\Rn \to \R$ defined
as 
$$
f(\bx,y)\ee y\mm \frac{1}{h(\bx)}.
$$
We have $I(f)\ee \emptyset$ and $f$ has no critical point. 
Observe that the partial derivatives of $f$ 
are uniformly bounded over $\Rn$, thus $f$ is a Lipschitz function over $\Rn$.

For $c\in \R$ define the following mapping  
$$
G: \Rn \to \R\times f^{\mm 1}(c), \;\; (\bx,y) \mapsto \left( f(\bx,y), 
\left(\bx, c + \frac{1}{h(\bx)} \right)\right).
$$
It is a Lipschitz homeomorphism with inverse 
$$
G^{\mm 1}\left(t, \left(\bx,c + \frac{1}{h(\bx)}\right)\right)
\ee \left(\bx,t + \frac{1}{h(\bx)}\right).
$$
The inverse $G^{\mm 1}$ is also Lipschitz, thus each value $c$ of $\R$ is 
Lipschitz trivial for $f$. Since any level of $f$ is a graph over 
$\Rn$, the function $f$ cannot be proper at $c$.  Last, there exists no vector
$\bv$ of $\Rn\setminus \bbo$ such that $\dd_\bv f \equiv 0$.
\end{proof}

\begin{rk}
The function $f$ defined in the proof of  
Proposition~\ref{propexampleRationalLipchitzManyVar}
is regulous \cite{KoNo,Kuc,4guys}. 
\end{rk}



\end{document}